\documentclass{article}

\usepackage{booktabs}       % professional-quality tables
\usepackage{amsfonts}       % blackboard math symbols
\usepackage{nicefrac}       % compact symbols for 1/2, etc.
\usepackage{microtype}      % microtypography

\usepackage{fncylab,enumerate,wrapfig,placeins}
\usepackage{bbm}
\usepackage{graphicx}
\usepackage{enumitem}  
\usepackage{amsmath,amsthm,amsfonts,amssymb,xcolor}  
\usepackage{pifont}

\usepackage{geometry}
\usepackage{leftindex}
\pdfoutput=1  %  for arxiv

\usepackage{hyperref}       % hyperlinks
\usepackage{url}            % simple URL typesetting

\usepackage{graphicx}  
\usepackage{relsize} % mathlarger
\usepackage{accents}  %just four doublehat 

\usepackage{caption}
\usepackage{textgreek,upgreek,bm}

\usepackage{titlesec}
\makeatletter
\newcommand\gobblepars{%
    \@ifnextchar\par%
 {\expandafter\gobblepars\@gobble}%
{}}
\makeatother

\def\whampar#1{\smallbreak\pagebreak[3]%
\noindent\textbf{#1} \gobblepars  \ \
\nopagebreak[4]%
}

\def\whamrm#1{\smallbreak\pagebreak[3]%
\noindent\text{\rm#1}\ \ \gobblepars}

\def\whamit#1{\smallbreak\pagebreak[3]%
\noindent\textit{#1}\ \ \gobblepars}

\def\wham#1{\smallbreak\pagebreak[3]%
\noindent\textbf{#1}\ \ \gobblepars}

\def\whamb{\wham{$\bullet$}}

\newcounter{rmnum}
\newenvironment{romannum}{\begin{list}{{\upshape (\roman{rmnum})}}{\usecounter{rmnum}
\setlength{\leftmargin}{2pt}
\setlength{\rightmargin}{4pt}
\setlength{\itemsep}{1pt}
\setlength{\itemindent}{5pt}
}}{\end{list}}

\newcounter{anum}

\def\tqsaprobe{\hbox{\scriptsize$\upxi$}}
\def\Sigmaqsa{\Sigma_{\tqsaprobe}}

\def\Deltaf{\gamma_f}
\def\tilDeltaf{\widetilde\gamma_f}
\def\Deltabarf{\bar\gamma_f}

\def\state{{\sf X}}

\def\MC{\text{\MC}}

\def\transpose{\intercal}

\def\limsup{\mathop{\rm lim\ sup}}

\def\argmin{\mathop{\rm arg\, min}}

\newcommand{\field}[1]{\mathbb{#1}}

\def\Re{\field{R}}

\def\bdd#1{b^{\text{\rm\tiny\ref{#1}}}}

\def\varble{\,\cdot\,}

\def\epsy{\varepsilon}

\def\eqdef{\mathbin{:=}}

\def\zsigma{\varsigma}

\def\ind{\mathbbm{1}}  %\bbbone}

\newtheorem{theorem}{Theorem}[section]

\newtheorem{proposition}[theorem]{Proposition}
\newtheorem{lemma}[theorem]{Lemma}

\usepackage{cleveref}
 \def\sbullet{{\scalebox{0.75}{\textbullet}}}

\Crefname{corollary}{Corollary}{Corollaries}
\Crefname{eqnarray}{eq.}{eqs.}
\Crefname{equation}{eq.}{eqs.}

\def\ctr{\text{\tiny{\sf ctr}}}
\def\opt{\text{\tiny\sf opt}}

\Crefname{figure}{Fig.}{Figs.}
\Crefname{tabular}{Tab.}{Tabs.}
\Crefname{table}{Tab.}{Tabs.}
\Crefname{lemma}{Lemma}{Lemmas}

\Crefname{theorem}{Thm.}{Thms.}
\Crefname{definition}{Definition}{Definitions}
\Crefname{section}{Section}{Sections}
\Crefname{proposition}{Prop.}{Propositions}
\Crefname{assumption}{Assumption}{Assumptions}
\Crefname{example}{Example}{Examples}

\def\trace{\hbox{\rm trace\,}}  
\def\bfmW{\bfmath{W}}

\def\tilpi{{\tilde \pi}}

\def\tilg{\tilde g}

\DeclareFontFamily{U}{mathx}{\hyphenchar\font45}
\DeclareFontShape{U}{mathx}{m}{n}{<-> mathx10}{}
\DeclareSymbolFont{mathx}{U}{mathx}{m}{n}
\DeclareMathAccent{\widebar}{0}{mathx}{"73}

\def\bfmath#1{{\mathchoice{\mbox{\boldmath$#1$}}%
{\mbox{\boldmath$#1$}}%
{\mbox{\boldmath$\scriptstyle#1$}}%
{\mbox{\boldmath$\scriptscriptstyle#1$}}}}

\def\bfmX{\bfmath{X}}

\def\bfPhi{\bfmath{\Phi}}

\def\nabobj{\raisebox{.1em}{\scalebox{0.75}{\scriptsize$\nabla\Obj$}}}

 \def\FRAC#1#2#3{\genfrac{}{}{}{#1}{#2}{#3}}

\def\ddt{{\mathchoice{\FRAC{1}{d}{dt}}%
{\FRAC{1}{d}{dt}}%
{\FRAC{3}{d}{dt}}%
{\FRAC{3}{d}{dt}}}}

\def\half{{\mathchoice{\FRAC{1}{1}{2}}%
{\FRAC{1}{1}{2}}%
{\FRAC{3}{1}{2}}%
{\FRAC{3}{1}{2}}}}

\def\clE{{\cal E}}

\def\clG{{\cal G}}

\def\clM{{\cal M}}

\def\clW{{\cal W}}

\def\cdG{c_{{\text{\lower1pt\hbox{d}}}} }

\usepackage{bbold}

\title{Global Convergence and Acceleration for
	\\
	Single Observation Gradient Free Optimization}

\author{Caio Kalil Lauand\thanks{Caio Kalil Lauand  is with the Division of Systems Engineering, Boston University, Boston, MA 02215, USA.
		Email: {\tt\small cklauand@bu.edu
		}
	}
	\and
	Sean Meyn\thanks{Sean Meyn is with the Department of Electrical and Computer Engineering, University of Florida, Gainesville, FL 32611, USA.
		Email: {\tt\small meyn@ufl.edu}
		\\
		Financial support from  
		ARO award
		W911NF2410389
		and NSF award  CCF 2306023
		is gratefully acknowledged.
	}
}

\begin{document}

\def\whampar#1{\smallbreak\pagebreak[3]%
\noindent\textbf{#1} \gobblepars  \ \
\nopagebreak[4]%
}

\newcommand*{\QED}{\hfill\ensuremath{\blacksquare}}
\def\qedIEEE{\nobreak\hspace*{\fill}~\QED\par \unskip }

\def\ProofOf#1{\whamit{#1}}
%The "wham" business is causing errors. Revert to below if issues persist
%\def\ProofOf#1{\smallbreak\noindent\textit{#1:} \  }

%%%%%%%%%%%%%%%%%%%%%%%%%%%%

%%%  QSA notation  

\def\tilnabla{{\widetilde{\nabla}\!}}
%
%\newcommand{\varqsa}{\theta}
%\def\varqsaPR{\varqsa^{\text{PR}}}

%%%%%%%%%%%  QSA  %%%%%%%%%%%  

\newcommand{\qsaprobe}{{\scalebox{1.1}{$\upxi$}}}  %\upzeta \textphi
\newcommand{\bfqsaprobe}{{\scalebox{1.1}{$\bm{\upxi}$}}}  
\newcommand{\qsaprobeTwo}{\qsaprobe^\bullet}  % 

\def\SAtime{\uptau}

%%%%%%%%%%%%   
\def\Lip{L}  % Lipschitz constant.   
 
\def\Obj{\Upgamma}  %      No good: \Uplambda , \Upomega , L
\def\barObj{\bar{\Obj}}
 
\def\Evo{\mathcal{X}} 

%%%%%%%%%%%%%%%%

 %  Bar
 
 %  Get rid of overline eventually
 
\def\barell{{\overline {\ell}}}

\def\bara{{\overline {a}}}
\def\barb{{\overline {b}}}
\def\barc{{\overline {c}}}
\def\bare{\bar{e}}

\def\barf{{\widebar{f}}}
\def\barfzap{\bar{f}^{\textup{\textsf{zap}}}}

\def\barg{{\widebar{g}}}

   \def\ubarc{\underline{c}}
   \def\ubarw{\underline{w}}
   \def\ubary{\underline{y}}

\def\barclE{\bar{\clE}}
\def\barcx{{\barc}_{\hbox{\it\tiny X}}}

\def\barcy{{\barc}_{\hbox{\it\tiny Y}}}

\def\bard{{\overline {d}}}

\def\barf{{\widebar{f}}}
\def\barg{{\widebar{g}}}

\def\barh{{\overline {h}}}
\def\bark{{\overline {k}}}
\def\barI{{\overline {l}}}
\def\barm{{\overline {m}}}
\def\barn{{\overline {n}}}

\def\barp{{\overline {p}}}
\def\barq{{\overline {q}}}
\def\barr{{\overline {r}}}
\def\bars{{\overline {s}}}
\def\barv{{\overline {v}}}
\def\barw{{\overline {w}}}
\def\barx{{\overline {x}}}
\def\bary{{\overline {y}}}

\def\barA{{\bar{A}}}
\def\barB{{\bar{B}}}
\def\barC{{\bar{C}}}
\def\barD{{\bar{D}}}
\def\barE{{\bar{E}}}
\def\barF{{\bar{F}}}
\def\barH{{\bar{H}}}
\def\barJ{{\bar{J}}}
\def\barL{{\bar{L}}}
\def\barP{{\bar{P}}}
\def\barQ{{\bar{Q}}}
\def\barR{{\bar{R}}}
\def\barS{{\bar{S}}}
\def\barT{{\bar{T}}}
\def\barU{{\bar{U}}}
\def\barX{{\bar{X}}}
\def\barY{{\bar{Y}}}
\def\barZ{{\bar{Z}}}
 
 \def\bx{{{\cal B}(\state)}}

\def \barzeta{{\overline{\zeta}}}

\def\baratom{{\overline{\atom}}}

\def\bareta{{\overline{\eta}}}
\def\barho{{\overline{\rho}}}
\def\barmu{{\overline{\mu}}}
\def\barnu{{\overline{\nu}}}

\def\bartheta{{\overline{\theta}}}

\def\barTheta{{\overline{\Theta}}}

\def\barpsi{{\bar{\psi}}}

\def\barnabla{\bar{\nabla}}

\def\barlambda{\bar \lambda}

\def\tilh{\tilde{h}}
\def\baralpha{{\bar{\alpha}}}
\def\barbeta{{\bar{\beta}}}

\def\barepsy{{\bar{\epsy}}}
 
\def\barbP{\overline{\bfPhi}}
\def\barPhi{\overline{\Phi}}
 \def\barsigma{\overline{\sigma}}

\def\barSigma{\overline{\Sigma}}
\def\barDelta{\overline{\Delta}}
\def\barGamma{\overline{\Gamma}}
\def\barLambda{\overline{\Lambda}}

\def\bartau{\overline{\tau}}

%  Hat

\def\haPsi{{\widehat{\Psi}}}
\def\haGamma{{\widehat{\Gamma}}}
\def\haSigma{{\widehat{\Sigma}}}
\def\habfPhi{{\widehat{\bfPhi}}}

\def\haPi{{\widehat \Pi}}

\def\haP{{\widehat P}}
\def\haZ{{\widehat Z}}

\def\haf{{\hat f}}
\def\hag{{\hat g}}
\def\hah{{\hat h}}
\def\ham{{\hat m}}
\def\hay{{\widehat y}}
\def\hav{\hat v}
\def\hatheta{{\hat\theta}}
\def\hapsi{{\hat\psi}}
\def\hanu{{\hat\nu}}
\def\hamu{{\hat\mu}}
\def\halpha{{\hat\alpha}}
\def\habeta{{\hat\beta}}
\def\hagamma{{\hat\gamma}}

\def\haphi{\hat{\phi}}

\def\hanabla{\widehat{\nabla}}

\def\halambda{{\hat\lambda}}
\def\haLambda{{\hat\Lambda}}

\def\haA{\widehat A}
\def\haK{\widehat K}

\def\haT{\widehat T}
\def\haX{\widehat X}

\def\barfalpha{\bar{f}_{\!\alpha}}
\def\Aalpha{\derbarf_{\!\alpha}}

\def\Athree{\text{\rm(A3${}^\circ$)}}
\def\AthreeV{\text{\rm(A3${}^\bullet$)}}
\def\derbarf{\bar{A}} 

%%%  RL notation 

\def\BelErr{\mathcal{B}}
\def\ocp{*} 

%  autocorrelation

 \def\Rtheta{R^\tTheta}

\def\REXP{R^\tEXP}

 \def\bEXP{b^\tEXP} 
 \def\btheta{b^\tTheta}

\def\SigmaCLT{\Sigma_{\text{\tiny \sf CLT}}}

\newlength{\noteWidth}
\setlength{\noteWidth}{.75in}
\long\def\notes#1{\ifinner
{\footnotesize #1}
\else 
\marginpar{\parbox[t]{\noteWidth}{\raggedright\tiny#1}}  %\footnotesize
\fi\typeout{#1}}

%%%%%%%%%%%%%%%%%%%%%%%%%%%%   

\def\psisub#1{\psi_{(#1)}}
\def\upsisub#1{\underline{\psi}_{(#1)}}
\def\tilpsisub#1{{\widetilde \psi}_{(#1)}}

\def\csub#1{c_{#1}}
\def\tilcsub#1{{\tilde c}_{#1}}

%%%%%%%%%%%%%%%%%%%%%%%%%%%%   

\def\EXPs{N}
\def\EXP{\mathcal{W}}

\def\tEXP{{\text{\tiny$\EXP$}}}

 \def\feeUnique{\mathcal{C}^\tTheta}

 \def\ctheta{c^\tTheta}
\def\cEXP{c^\tEXP}

\def\tERM{{\text{\tiny ERM}}}

\def\density{p}

\def\RateFn{\bar{I}}

\def\diffDPT{D}

\def\Amap{a}

\def\BE{{\cal B}}
\def\maxBE{\overline{\cal B}}

\def\Tdiff{\mathcal{D}}
\def\barTdiff{\bar{\mathcal{D}}}
\def\tilTdiff{\widetilde{\mathcal{D}}}

\def\thetaPR{\theta^{\text{\tiny\sf  PR}}}
\def\tilthetaPR{\tilde{\theta}^{\text{\tiny\sf  PR}}}

\def\lambdamin{\lambda_{\textup{\rm\tiny min}}}
\def\lambdamax{\lambda_{\textup{\rm\tiny max}}}

\def\preODEstate{\Uppsi} %\Upxi BAD %%    \Upupsilon   
\def\bfpreODEstate{\bm{\preODEstate}}
\def\barpreODEstate{\bar{\preODEstate}}
\def\tilpreODEstate{\tilde{\preODEstate}}
\def\hatpreODEstate{\widehat{\Upxi}}
\def\bfhatpreODEstate{\hat{\bfpreODEstate}}

\def\ODEstate{\Uptheta} %\xi
\def\bfODEstate{\bm{\Uptheta}}
\def\barODEstate{\widebar{\Uptheta}}
\def\tilODEstate{\widetilde{\Uptheta}}
\def\haODEstate{\widehat{\Uptheta}}
\def\bfhaODEstate{\widehat{\bfODEstate}}

\def\ODEstatePR{\ODEstate^{\text{\tiny\sf PR}}}
\def\ODEstatePRF{\ODEstate^{\text{\tiny\sf PRF}}}
\def\ODEstatePRm{\ODEstate^{\text{\tiny\sf PR$-$}}}

\def\odestate{\upvartheta}
\def\bfodestate{\bm{\upvartheta}}

\def\Nsam{N}

\def\belief{\text{\large$\beta$}}

\def\fee{\upphi}
\def\feex{\widetilde{\fee}}

 \def\feeF{\phi^{\hbox{\tiny \rm F}}}
 \def\feeD{\phi^{\hbox{\tiny \rm D}}}
 
\def\Hor{\mathcal{T}} %H}
\def\hor{\uptau}   %\iota}

\def\Qstar{Q^\star}
\def\Hstar{H^\star}

\def\bfpi{\bfmath{\uppi}} 
\def\cpi{\check{\uppi}}
\def\tilpi{{\tilde \uppi}}

\def\piM{\uppi^{\hbox{\tiny \rm M}}}

\def\rglCLT{\rangle_{\text{\tiny \sf CLT}}}
\def\rglLt{\rangle_{L_2}}

\def\prpi{\upmu}

%%%   CDC 2023

\def\hatDelta{\widehat \Delta}
\def\MD{\clW}
\def\Mart{\clM}

\def\elig{\zeta}
\def\bfelig{\bfmath{\zeta}}

\def\uc{\underline{c}}

\def\uh{\underline{h}}
\def\uq{\underline{q}}

\def\uH{\underline{H}}
\def\uQ{\underline{Q}}
\def\uQstar{\underline{Q}^\star}

\def\utheta{\underline{\theta}}  
  
\def\discRate{\Gamma}
\def\disc{\gamma}

\def\stepf{\beta}

%%%

\def\scerrorSymbol{Z}
  
\def\scerror#1#2{\scerrorSymbol_{#1}^{(#2)}}

\def\scerrorpull#1#2{\scerror{#1 \mid \mydash}{#2}}

\def\sclim{X}

\def\cz{\check{z}}
\def\tilz{\tilde{z}}

\def\barfinf{\barf_{\infty}}
\def\SAtime{\uptau}
\def\SAfn{s}

\def\tTheta{{\text{\tiny$\Theta$}}}

\def\SigmaTheta{\Sigma_\tTheta}

\def\thetaPR{\theta^{\text{\tiny\sf  PR}}}
\def\tilthetaPR{\tilde{\theta}^{\text{\tiny\sf  PR}}}
\def\SigmaPR{\Sigma^{\text{\tiny\sf PR}}_\tTheta}

%%%%%%%%%%%%%%%%%%%%%%%%%%%%% MISCELLANEOUS

\def\bcbm{b_{\hbox{{\rm\tiny CBM}}}}
 
\def\Gexp{\clG^+(\gamma)}
\def\Gexps{\clG^+(s)}

 \def\head#1{\paragraph{\textit{#1}}}

 \def\bfDelta{\bfmath{\Delta}}

\def\implies{\Rightarrow}

\def\eqdef{\mathbin{:=}}

\def\lebmeas{\mu^{\hbox{\rm \tiny Leb}}}
\def\Prbty{{\cal  M}}
\def\Prob{{\sf P}}
\def\Probsub{{\sf P\! }}
\def\Expect{{\sf E}}

\def\Var{\hbox{\sf Var}}
\def\Cov{\hbox{\sf Cov}}
\def\CV{\hbox{\sf CV}}

\def\lgmath#1{{\mathchoice{\mbox{\large #1}}%
{\mbox{\large #1}}%
{\mbox{\tiny #1}}%
{\mbox{\tiny #1}}}}

\def\One{{\mathchoice{\lgmath{\sf 1}}%
{\mbox{\sf 1}}%
{\mbox{\tiny \sf 1}}%
{\mbox{\tiny \sf 1}}}}

\def\Zero{{\mathchoice{{\sf 0}}%  \lgmath
{\mbox{\sf 0}}%
{\mbox{\tiny \sf 0}}%
{\mbox{\tiny \sf 0}}}}

\def\taboo#1{{{}_{#1}P}}

\def\tinybull{\hbox{\tiny$\ddag$}}

\def\bullP{{{}_{\tinybull}P}}

\def\bullX{X_{\tinybull}}
\def\bfmbullX{\bfmX_{\tinybull}}
\def\bullstate{\state_{\tinybull}}
\def\bullx{x^{\tinybull}}
\def\bullh{h_{\tinybull}}
\def\bullc{c_{\tinybull}}
\def\bulltau{\tau_{\tinybull}}

\def\io{{\rm i.o.}}

\def\as{{\rm a.s.}}
\def\Vec{{\rm vec\, }}

 \def\Tol{\text{\rm Tol}}

 \def\epsy{\varepsilon}

\def\varble{\,\cdot\,}

%%%

%%%%%%%%%%%%%%%%%%%%%%%%%%%%%%%%%%%%%%%%%%
%  Stuff for the Appendix

\def\uE{\underline{E}}

\def\barfzap{\bar{f}^{\textup{\textsf{zap}}}}

\def\barfTD{\bar{f}^{\textsf{TD}}}
\def\ATD{A^{\textsf{TD}}}
\def\bTD{b^{\textsf{TD}}}

\def\barfg{\bar{f}^{\textsf{g}}}
\def\Ag{A^{\textsf{g}}}
\def\bg{b^{\textsf{g}}}

\def\barfG#1{\barf^{#1}}

%%%%%

\def\formtmp#1#2{{\vskip12pt\noindent\fboxsep=0pt\colorbox{#1}{\vbox{\vskip3pt\hbox to \textwidth{\hskip3pt\vbox{\raggedright\noindent\textbf{#2\vphantom{Qy}}}\hfill}\vspace*{3pt}}}\par\vskip2pt%
\noindent\kern0pt}}

%%%%%%%%%%%%%%%%%%%

\maketitle
\thispagestyle{empty}
%\pagestyle{empty}

%%%%%%%%%%%%%%%%%%%%%%%%%%%%%%%%%%%%%%%%%%%%%%%%%%%%%%%%%%%%%%%%%%%%%%%%%%%%%%%%
\begin{abstract}

Simultaneous perturbation stochastic approximation (SPSA) is an approach to gradient-free optimization introduced by Spall as a simplification of the approach of 
Kiefer and Wolfowitz.    In many cases the most attractive option is the single-sample version known as 1SPSA, which is the focus of the  present paper, containing   two major contributions:  a modification of the algorithm designed to ensure convergence from arbitrary initial condition, and a new approach to exploration to dramatically accelerate the rate of convergence.    Examples are provided to illustrate the theory, and to demonstrate that estimates from unmodified 1SPSA may diverge even for a quadratic objective function.

\end{abstract}

\tableofcontents

\clearpage
 
\section{Introduction}

This article addresses algorithms for estimating the global minimum of an objective function $\Obj: \Re^d \to \Re$, denoted by $\theta^\opt \in \argmin_\theta \Obj(\theta)$. In many practical scenarios, the gradient of $\Obj$ is either unavailable or computationally expensive to obtain. To tackle such challenges, zeroth-order (or derivative-free) optimization methods are available, relying solely on function evaluations rather than gradient information.

Theory is devoted to two special cases of the Simultaneous Perturbation Stochastic Approximation (SPSA) approach introduced by Spall \cite{spa03,spa97,spa92}: 
\begin{subequations}
\begin{align}
\hspace{-0.32cm} 
\text{1SPSA:} \  \theta_{n+1} 
&= 
 \theta_{n} - \alpha_{n+1}   \frac{1}{\epsy_n}  \qsaprobe_{n+1} \Obj^+_n   
\label{e:1SPSA}
\\
\hspace{-0.32cm} 
\text{2SPSA:}\ \theta_{n+1} 
&= 
 \theta_{n} - \alpha_{n+1}   \frac{1 }{2 \epsy_n} \qsaprobe_{n+1} [\Obj^+_n - \Obj^-_n]    
\label{e:2SPSA}
 \\[.5em]
\Obj^+_n = \Obj(\theta_n & + \epsy_n\qsaprobe_{n+1})  \, ,  
\quad 
\Obj^-_n =\Obj(\theta_n - \epsy_n\qsaprobe_{n+1})
\end{align}
in which $\{\alpha_{n+1}\}$ is known as the ``step-size'' or ``learning-rate'',
$\{ \epsy_n \}$  is the  \textit{exploration gain}, and   $\{ \qsaprobe_{n+1}   : n\ge 0 \}$
 is the 
\textit{exploration sequence}---all  
   specified in algorithm design.
\label{e:SPSA_spall}
\end{subequations}

The 1SPSA recursion  is  attractive because it requires only a single observation at each iteration, and hence is more suitable than its two measurement counterpart when there are noisy observations  (that is, for any $\theta$,  the observation $ \Obj(\theta)$ is subject to noise).   
In view of this and space constraints,  we consider exclusively 1SPSA in the analysis that follows.   
While the theory presented here may be extended to the case of noisy observations, we assume throughout that observations are noise-free.

%I

It is assumed in \cite{spa03} that the entries of $\{  \qsaprobe_{n+1}  \}$ are i.i.d.\  with a symmetric distribution. Analysis is based on vanishing $\{\upepsilon_n \}$ and vanishing step-size $\{ \alpha_n\}$ to allow for almost sure convergence to $\theta^\opt$, along with bounds on the rate of convergence.  
However,  in this prior work it is \textit{assumed} that the   sequence of estimates $\{ \theta_n \}$ is bounded with probability one.

Given the relative maturity of stability theory for stochastic approximation,   this boundedness assumption is easily verified for 2SPSA under the standard assumptions that $\nabla \Obj$ is Lipschitz continuous and its norm is coercive  \cite[Ch. 4]{CSRL};  for example, to establish boundedness one can apply the ODE@$\infty$ approach of \cite{bormey00a,borchedevkonmey25}.     Stability theory for 1SPSA is far more challenging because the recursion  typically violates the crucial Lipschitz condition imposed in the SA literature;   Lipschitz continuity fails even for a quadratic objective.

A very similar challenge was discussed in \cite{laumey22d} in the context of  extremum seeking control (ESC).    In this simpler continuous-time setting it was possible to prove that ESC has finite escape time when applied to a positive definite quadratic objective, and the initial condition is sufficiently large.   We do not have a proof of divergence in this discrete time setting, but we illustrate the potential for instability through numerical experiments.

\whamit{Divergence of estimates for 1SPSA.} Consider the simplest example with $d=1$ and quadratic 
 objective $\Obj(\theta) = \theta ^2$, so that  $\theta^\opt = 0$. 
Both versions of SPSA in \eqref{e:SPSA_spall} were implemented with $\bfqsaprobe$ a scaled and shifted Bernoulli sequence taking values in  $\{-1,1\}$ with $\Prob(\qsaprobe = 1) = \half$. The step-size was $\alpha_n = n^{-0.6}$ and two exploration gains were tested $\epsy_n = n^{-0.3}$ (oblivious exploration) and $\epsy_n \equiv \upepsilon(\theta_n) = n^{-0.3}\sqrt{1+ | \theta_n|^2} $ (active exploration);  the first choice is consistent with Spall's  requirement that $ \sum_{i=0}^\infty \alpha^2_n/\epsy^2_n < \infty$ \cite[Ch. 7]{spa03}.

\begin{figure}[h]
	\centering
\includegraphics[width= 0.8\hsize]{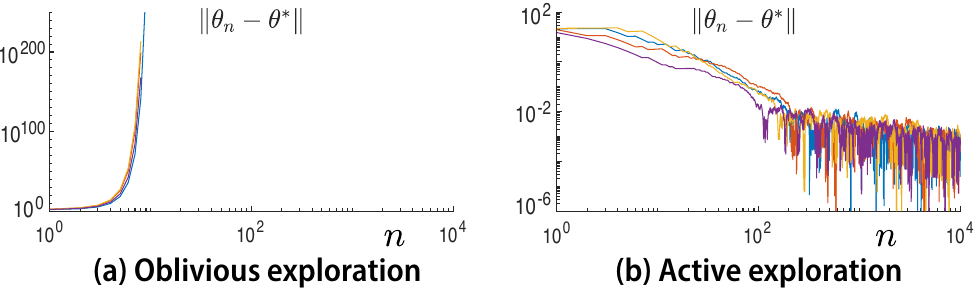}
\caption{Failure for four independent runs of 1SPSA.}
\label{fig:quad}
\end{figure}

\Cref{fig:quad} shows the evolution of estimates from four independent runs of each algorithm, differing by the initial condition  satisfying $ |\theta_0| \leq 10$.   
In \Cref{fig:quad}~(a) we observe that estimates from 1SPSA are unbounded with a constant exploration gain;  that is,  $\epsy_n \equiv \upepsilon_\sbullet>0$ (oblivious exploration).    Clarifying the source of instability is one of the contributions of this paper. 

Results obtained using  $\epsy_n   =  \upepsilon_\sbullet \sqrt{1+ | \theta_n|^2} $ (active exploration), with identical initial conditions, are shown in \Cref{fig:quad}~(b).    
The   plots are consistent with the conclusions of    \Cref{t:1SPSAconv},  that the algorithm  is convergent from any initial condition with active exploration.

\wham{Approach to analysis:} Analysis of SPSA starts with the recognition that these algorithms are special cases of stochastic approximation (SA). 
The framework of SA revolves around solving the   root-finding problem $\barf(\theta^*)=0$, 
in which   the function $\barf: \Re^d \to \Re^d$ is defined as the expectation
\[
\barf(\theta) = \Expect[f(\theta,\Phi)] \,,\quad \textit{in which $\Phi$ is a random vector.
}
\]

The general SA algorithm is a $d$-dimensional recursion:  with initialization $\theta_0 \in \Re^d$,
\begin{equation}
\theta_{n+1} = \theta_n + \alpha_{n+1} f(\theta_n,\Phi_{n+1})
\label{e:SA_recur}
\end{equation}
where  $\bfPhi = \{\Phi_n\}$ is a sequence of random vectors converging in distribution to $\Phi$ as $n\to \infty$.    The so-called ODE method for establishing convergence of  $\{\theta_n\}$ involves first establishing global asymptotic stability of the associated   \textit{mean flow}:
\begin{equation}
\ddt \odestate_t  = \barf(\odestate_t)
\label{e:meanflow}
\end{equation}

%E
%\end{equation}

\wham{Contributions:}    The following conclusions are obtained subject to   the standard assumptions that $\nabla \Obj$ is Lipschitz continuous and its norm is coercive, along with assumptions on exploration and step-size.

\whamb  A parameter-dependent exploration gain is introduced to ensure global stability of 1SPSA, in the sense of ultimate boundedness: 
there is a fixed constant $b_\tTheta$ such that for any initial condition $\theta_0 \in \Re^d$, $\limsup_{n \to \infty}\| \theta_n\|\leq b_\tTheta$ with probability one.   Under additional assumptions we establish convergence of the algorithm.

\begin{subequations}

The gain is expressed  $\epsy_n = \upepsilon(\theta_n)$ for a smooth function $\upepsilon\colon\Re^d\to (0,\infty)$.     For ease of presentation, analysis is restricted to either of these two possibilities:  
\begin{align}
\upepsilon(\theta) 
&=
 \upepsilon_\sbullet\sqrt{1+ \| \theta - \theta^\ctr\|^2 / \upsigma_p^2 }
\label{e:upepsy2}
\\
\upepsilon(\theta) 
&= 
\upepsilon_\sbullet \sqrt{1 + \Obj(\theta) - \Obj^-} 
\label{e:upepsy1}
\end{align}
where $\upepsilon_\sbullet>0$ is fixed.
In \eqref{e:upepsy2} the vector  $\theta^\ctr$ is interpreted as an a-priori estimate of $\theta^\opt$ with $\upsigma_p$ quantifying uncertainty.    
It is assumed    in \eqref{e:upepsy1} 
that $ \Obj(\theta) \ge \Obj^-$ for all $\theta$.

\textit{The value of a state-dependent exploration gain goes far beyond stabilization of  the algorithm}.   Consider its role as a technique to explore more efficiently:       if $\Obj(\theta_n)$ is large,      a larger  exploration gain ensures that the estimates move more    rapidly away from this undesirable parameter value.

\label{e:upepsy}
\end{subequations}

\whamb   

For either choice  of modified exploration gain  of the form \eqref{e:upepsy}, we show in \Cref{t:1SPSAconv} that
the mean flow  \eqref{e:meanflow} associated with 1SPSA  is exponentially asymptotically stable (EAS)
to an equilibrium $\theta^*$,   provided the gradient flow $\ddt x_t  = - \nabla \Obj\, (x_t) $  is   EAS,  and $\upepsilon_\sbullet>0$ is sufficiently small.  
Moreover, we establish that the 1SPSA algorithm is convergent to $\theta^*$  in this case.
  There is bias, satisfying the order bound $\| 
\theta^\opt - \theta^* \| = O(\upepsilon_\sbullet^2)$.

\whamb 
Even when convergent, the algorithm with or without modified exploration gain may suffer from massive variance when $\upepsilon_\sbullet>0$ is small.  
It is argued that variance can be reduced dramatically through  the introduction of negatively correlated exploration.   
Analysis is restricted to one approach coined \textit{zig-zag} exploration, designed so its power spectral density evaluated at the origin  (also known as the asymptotic covariance) is
 \textit{zero}.    The impact surveyed in  \Cref{t:1SPSAcov}  is remarkable: 
  there are constants $b_0,b_1<\infty$ such that
  \begin{subequations}
\begin{align}
\textbf{i.i.d.:} \, \lim_{N\to\infty}  N \,  \trace( \Cov(\upbeta_N^{\barf} )  )    
& \ge 
 b_0 \frac{1}{\upepsilon_\sbullet^2}  \big[ \Obj\, (\theta^* ) \big]^2
\label{e:IID_var}
\\
\textbf{zig-zag:} \, \lim_{N\to\infty} N  \,  \trace( \Cov(\upbeta_N^{\barf} ) )   
  & \le 
    b_1 \upepsilon_\sbullet^2
\label{e:ZZ_var}
\end{align}
\label{e:BigSmallTargetBias}
\end{subequations}
in which $\upbeta_N^{\barf} $ denotes the  \textit{empirical target bias}:
\begin{equation}
\upbeta_N^{\barf}  \eqdef   \frac{1}{N}  \sum_{n=0}^{N-1}  \barf\, (\theta_n)
\label{e:target_bias}
\end{equation}
\Cref{fig:quad_sincos} serves to illustrate these conclusions---details may be found in 
\Cref{s:SPSA_exp}.

\wham{Commentary:}  
It may be shown using recent theory that the mean-square rate of convergence of $\{\theta_n \}$ to $\theta^*$ is   of order $O(\alpha_n)$,
and this can be accelerated to $O(1/n)$ through the averaging technique of Polyak and Ruppert \cite{borchedevkonmey25}.

Application of theory from \cite{borchedevkonmey25} requires the use of a non-vanishing exploration gain.
While this  restriction was imposed in part for ease of analysis, there is also a practical reason:  while we predict that an unbiased algorithm can be obtained  through the replacement  of $\upepsilon_\sbullet$ by a vanishing sequence $\{ \upepsilon_n \}$ in \eqref{e:upepsy},  the rate of convergence
 in mean square
 will be slowed significantly to 
$O(n^{-\beta} )$ with $\beta<1$.   See the literature review for details.   
 
The bounds in \eqref{e:BigSmallTargetBias} are admittedly abstract.
A more attractive result would conclude that these bounds hold with \eqref{e:target_bias}
replaced by $\nabla \Obj\, (\theta_n)$ or its average.     While we do not yet have a proof of such bounds, in \Cref{t:Lip_qSGD}
 we establish that $\barf\approx - \Sigmaqsa \nabla \Obj$,   with $\Sigmaqsa$ the steady-state covariance    of $\qsaprobe_n$. 
   More precisely,  $ \barf(\theta)  = -\Sigmaqsa    \nabla \Obj(\theta)   +   \upepsilon_\sbullet^2 \Deltabarf(\theta) $, in which $\Deltabarf$ is Lipschitz.

The objective-dependent lower bound in 
\eqref{e:BigSmallTargetBias} suggests that performance using i.i.d.\ exploration might be improved by adding a constant to the objective, perhaps adaptively so that $\Obj (\theta^* )$ is small.    This may also improve the performance of 1SPSA using zig-zag exploration.

\wham{Literature Review:}
While the present paper focuses on algorithms introduced by Spall in \cite{spa87, spa92, spa97}, this work follows the seminal work of  Keifer and Wolfowitz  \cite{kiewol52};  see also \cite{bhaprapra13,bhabor03,gos15,asmgly07}.

\begin{figure*}[h]
\centering
	\includegraphics[width= 0.9\hsize]{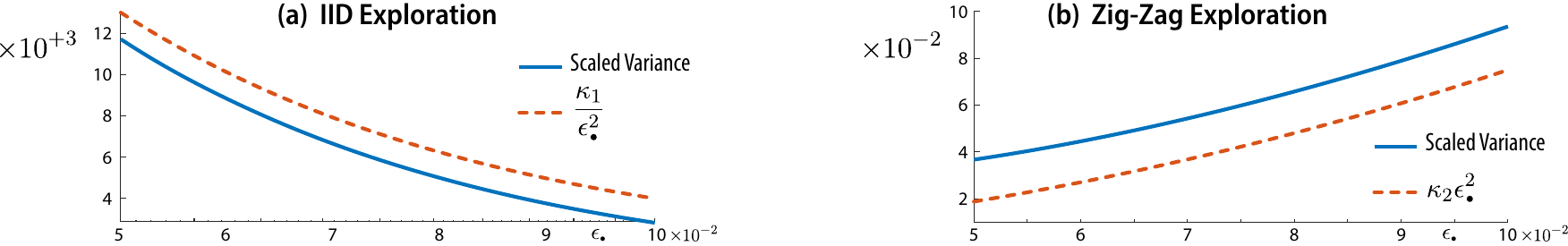}
	\caption{Impact of exploration correlation  when the exploration gain is small:    (a)  Using i.i.d.\ exploration the scaled variance grows as $1/\upepsilon_\sbullet^2$.
	(b)  Using zig-zag exploration the scaled variance \textit{vanishes} as $\upepsilon_\sbullet^2$.}
	\label{fig:quad_sincos}
\end{figure*}

There is substantial research in the vanishing probing gain setting:   the best possible convergence rate for the mean square error is  $O(n^{-\beta} )$ with  
$\beta =  (p-1)/p $,
 provided the  objective function is $p$-fold differentiable at $\theta^\ocp$ \cite{poltsy90}.   Upper bounds appeared earlier in \cite{fab68b}.
See \cite{dipren97,dip03,spa12,pasgho13,larmenwil19} for more recent history.

The incorporation of a state-dependent exploration gain was first proposed by the authors in \cite{laumey22d} to ensure stability of extremum-seeking control (ESC).
%\
This enabled application of general theory to obtain    sharp rates of convergence \cite{laumey22d,laumey22e,laumey25a}. 
%;   

ESC theory has been developed almost exclusively in continuous time.   
While the history of ESC predates SPSA, the  1SPSA recursion  with deterministic exploration 
		may be regarded as
	 an Euler approximation of the simplest ESC ODE.
 See also  \cite{bhafumarwan03} for a treatment of
2SPSA with deterministic exploration.

\section{Main Results}

\subsection{Preliminaries}

\wham{Assumptions:}
The following  assumptions are in place throughout:

\wham{(A1)} 
The step-size sequence is of the form $\alpha_n =  \min\{\alpha_0, n^{-\rho} \}$ with $\rho \in (1/2,1)$ and $\alpha_0>0$. %

\wham{(A2)}
$\Obj \colon\Re^d \to \Re$ is twice continuously differentiable,  with first and second derivatives globally Lipschitz continuous, and 
there is $\delta>0$ such that  $\|\nabla \Obj (\theta)\|  \ge \delta \| \theta\| $ for $\|\theta\| \ge \delta^{-1}$.

\wham{(A3)} The exploration sequence $\bfqsaprobe$ is a functional of a zero-mean i.i.d. sequence $\bfmW \eqdef \{ W_n: n \geq 0 \}$ evolving on a compact subset of $\Re^d$ and satisfying the following for each $n$:  
\[
\begin{aligned} 
\Expect[   W^i_n  W^j_n  W^k_n    ]
&=0 
\quad \text{ for all $1 \leq i \leq j \leq k \leq d$}
\\
\Sigma_W \eqdef  \Expect[W_n W^\transpose_n ] ) &>0
\end{aligned}
\]
Two cases are considered. For $n\geq1$,

\wham{(A3i)}   \textbf{i.i.d.}: $\bfqsaprobe_n = W_n$.

\wham{(A3ii)}   \textbf{zig-zag}: $ \qsaprobe_{n}  =     \zsigma[W_n - W_{n-1}]$ with $\zsigma>0$.

\wham{(A4)}  The gradient flow ODE $\ddt x_t = - \nabla \Obj(x_t)$ is globally exponentially asymptotically stable, with $\theta^\opt \in\Re^d$ its unique stationary point.
Moreover, $\nabla \Obj\, (\theta^\opt)$ is strictly positive definite.

Under (A2) there exists at least one global minimizer, which is denoted $\theta^\opt$; (A4) implies that $x_t \to \theta^\opt$ as $t \to \infty$.

\wham{Notation:}

For a  $d$-dimensional vector-valued random variable $X$ and $p\ge 1$,
the $L_p$ norm is denoted 
$\| X \|_p = ( \Expect[\| X \|^p])^{1/p}$,  and the $L_p$  \textit{span norm}  $ \|X  \|_{p,s}= \min\{ \| X - c \|_p  : c \in \Re^d \}$.   
When $p=2$ we have  
\begin{align*}
\|X  \|_{2,s} & = \sqrt{\trace(\Cov(X))}   
%\label{e:spandef}
\\
\|X  \|_{2}^2 & = \|X  \|_{2,s}^2  +  \| \Expect[X] \|^2   
%\label{e:L2andSpan}
\end{align*}

\subsection{Design Guidelines for SPSA}

\wham{Stabilization through exploration gain} 
The proofs of the two propositions that follow are contained in    \Cref{s:Proofs!}.

\begin{proposition}
\label[proposition]{t:1SPSAconv}   If (A1)--(A2) hold along with either version of (A3), then 
there is a fixed constant $b_\tTheta$ such that with probability one,
\begin{equation}
\limsup_{n \to \infty}\| \theta_n\|\leq b_\tTheta  \quad \textit{for any initial condition $\theta_0 \in \Re^d$}
\label{e:UltBdd}
\end{equation}
If in addition (A4) holds, then there exists $\upepsilon_\sbullet^0>0$ such that the following hold for $\upepsilon_\sbullet \in (0, \upepsilon_\sbullet^0 ]$:

\begin{romannum}

\item   The algorithm is convergent: There is a unique $\theta^*\in\Re^d$ satisfying $\barf(\theta^*) =0$,  and for each initial condition
 $\theta_n \to \theta^*$ almost surely and in mean-square.

\item   The order bound 
  $\| \barf(\theta^\opt) \| \le   O(\upepsilon_\sbullet^2) $  holds for  $0\le \upepsilon_\sbullet \le  \upepsilon^0_\sbullet$.    
Moreover, 
$\| \theta^* - \theta^\opt \| =  O(\upepsilon_\sbullet^2) $ provided   $ \nabla^2   \Obj (\theta^\opt)$ is  positive-definite.

\end{romannum}
\end{proposition}

\wham{Acceleration through exploration} 
In either version of assumption (A3), it follows  that the polynomial moments of $\bfqsaprobe$ are finite

The next result reveals the remarkable value of zig-zag exploration. The span norm \textit{vanishes} quadratically with $\upepsilon_\sbullet$ when $\qsaprobe$ is zig-zag, rather than diverging to infinity in the i.i.d. case.

\begin{proposition}
\label[proposition]{t:1SPSAcov}   Suppose (A1), (A2) and (A4) hold.
\begin{romannum}
\item If in addition (A3i) holds, there is $\alpha_0>0$ , $\upepsilon_\sbullet^0>0$  and $b_0<\infty$ such that    for $\alpha \in (0,\alpha_0]$  and $\upepsilon_\sbullet \in (0, \upepsilon_\sbullet^0)$,
the target bias  \eqref{e:target_bias} satisfies \eqref{e:IID_var}.

\item If in addition (A3ii) holds, there is $\alpha_0>0$, $\upepsilon_\sbullet^0>0$   and $b_1<\infty$ such that such that    for $\alpha \in (0,\alpha_0]$  and $\upepsilon_\sbullet \in (0, \upepsilon_\sbullet^0)$, 
the target bias  \eqref{e:target_bias} satisfies \eqref{e:ZZ_var}.
\end{romannum}
\end{proposition}

\section{Numerical Experiments}
\label{s:SPSA_exp}

The experiments for which results are surveyed next aim to illustrate the conclusions of \Cref{t:1SPSAcov}.

\wham{Simulation Setup:} The 1SPSA algorithm was used to minimize the objective $\Obj(\theta) = \theta^2 - \cos(\theta)  - \sin(5 \theta)/5 + 4$, in which $d=1$. 
The trigonometric terms are introduced for two reasons:  so that $\nabla\Obj\, (\theta^\opt )$ has a unique solution, $\theta^\opt =0$,  and so that the objective is not symmetric around $\theta^\opt$ (which would imply $\theta^* = \theta^\opt$).

Both i.i.d.\ and zig-zag exploration were considered, constructed based on a common i.i.d. sequence $\{W_n\}$,  in which $W_n$ was uniformly sampled from $[-1,1]$ for each $n$. 
For i.i.d.\ exploration $\qsaprobe_n = W_n$.    For zig-zag,   $ \qsaprobe_{n}  =     \zsigma[W_n - W_{n-1}]$ with  $\zsigma = 1/\sqrt{2}$   selected so that the variance of $\qsaprobe_n$ is the same as in the i.i.d.\ case.   

The step-size sequence was $\alpha_n = n^{-0.6}$ and the exploration gain was chosen state-dependent,
 following \eqref{e:upepsy2} with $\theta^\ctr =0$ and $\upsigma_p =1$. The algorithm was implemented for several choices of  $\upepsilon_\sbullet$ in the range $ [0.05,0.1]$. For each value of $\upepsilon_\sbullet$, $M=100$ independent experiments were carried out for a time horizon of $N = 5 \times 10^5$ with initial conditions $\{\theta_0^i : 1 \leq i \leq M \}$ uniformly sampled from $[-10,10]$. 

For the $i^{\text{th}}$ run with $\upepsilon_\sbullet$ fixed, the empirical mean of the gradient  was computed:  
\[
\upbeta_{N,\upepsilon_\sbullet}^{\nabobj,i} 
= 
\frac{1}{N-N_0+1}\sum_{k=N_0}^N \nabla \Obj(\theta_k)
\, , \quad 1 \leq i \leq M
\]
in which $N_0 = 1.5 \times 10^5$,  chosen to reduce the impact of transients \cite{CSRL}.

\wham{Results:} \Cref{fig:quad_sincos} depicts the evolution of $(N-N_0) \Cov( \upbeta_{N,\upepsilon_\sbullet}^{\nabobj} )$ (i.e., the scaled variance of the mean of the gradient across independent runs) as a function of $\upepsilon_\sbullet$. 

Also shown in \Cref{fig:quad_sincos} for comparison with the bounds expected by \Cref{t:1SPSAcov} are  the functions $r_1(\upepsilon_\sbullet) = \kappa_1/\upepsilon_\sbullet^2 $ and $r_2(\upepsilon_\sbullet) = \kappa_2 \upepsilon_\sbullet^2 $. The constants $\kappa_1$ and $\kappa_2$ were chosen to aid comparison.

The value of zig-zag exploration  is illustrated once more by the plots in \Cref{fig:quad_sincos}. As  expected from the results in \Cref{t:1SPSAcov} and the fact that $\barf \approx -\nabla \Obj$ from \Cref{t:Lip_qSGD}, we see that the variance vanishes with $\upepsilon_\sbullet$ when the exploration is zig-zag, while growing without bound in the i.i.d. case.
 
\section{Conclusions}

With attention to recent SA theory we introduce here a gradient free optimization algorithm based on 1SPSA that is globally convergent, and with drastically reduced variance.  
Current research is devoted to non-convex settings in which the mean-flow is not globally asymptotically stable.

 \clearpage

\appendix
\section{Technical Proofs}

The 1SPSA recursion \eqref{e:1SPSA} can be expressed as SA  \eqref{e:SA_recur}, in which 
$\Phi_n = \qsaprobe_n$ in the i.i.d.\ setting (A3i),   and  $\Phi_n = (W_{n-1};W_{n})$ for zig-zag (A3ii).  
In either setting, it is convenient to adopt the following notation for the SA recursion,  
\begin{equation}
 \theta_{n+1} = \theta_n + \alpha_{n+1} [\barf(\theta_n) + \Delta_{n+1}]
 \label{e:Noisy_Euler}
 \end{equation}
 in which $\Delta_{n+1} = f(\theta_n,\Phi_{n+1}) - \barf(\theta_n)$. 
 This traditional change of notation is employed to justify the interpretation of \eqref{e:SA_recur} as a noisy  Euler approximation of the  mean flow \eqref{e:meanflow}.

The proofs of \Cref{t:1SPSAconv,t:1SPSAcov} require that we establish that the assumptions of \cite[Thm. 5]{borchedevkonmey25} hold.

\subsection{Lipschitz continuity} 
We begin with the essential Lipschitz continuity assumption for $f$.  
When convenient, to save space we write $\upepsilon_\theta$ instead of $\upepsilon(\theta)$.
Denote $f(\theta,\Phi) \eqdef -  \qsaprobe   \Obj(\theta +  \upepsilon_\theta   \qsaprobe )   / \upepsilon_\theta $ with associated  mean vector field $\barf (\theta) = \Expect[f(\theta,\Phi)]$.

\begin{proposition}
\label[proposition]{t:Lip_qSGD} 
Consider the 1SPSA algorithm in which $f$ is defined by \eqref{e:1SPSA}.  
Suppose that (A2) and either (A3i) or (A3ii) hold, and that $\upepsilon_\theta$ is defined by either choice in \eqref{e:upepsy}.
Then, $f$ is uniformly Lipschitz continuous:  there is $L>0$ such that
\begin{equation}
\|f(\theta',\qsaprobe) - f(\theta,\qsaprobe)   \|  \le L  \|\theta' - \theta \| \,, \quad  \theta \,, \, \qsaprobe \in\Re^d 
\label{e:fLip}
\end{equation} 
Moreover, for all $\theta, \qsaprobe\in\Re^d$,
\begin{subequations}
\begin{align}
&
\begin{aligned}
f(\theta,\Phi)  = 
 &
 -  \frac{1}{ \upepsilon_\theta}      \qsaprobe\Obj(\theta)  -      \qsaprobe \qsaprobe^\transpose   \nabla   \Obj\, (\theta )  
\\
& 
\, - \half \upepsilon_\theta    \qsaprobe^\transpose   \nabla^2   \Obj\, (\theta )   \qsaprobe \qsaprobe^\transpose   +  \Deltaf(\theta,\qsaprobe)
\end{aligned}
\label{e:f_Taylor}
\\[.5em]
& \barf(\theta)  =    -      \Sigmaqsa   \nabla   \Obj\, (\theta )   +  \Deltabarf(\theta)
\label{e:meanflow_fTaylor}
\end{align}
in which  $\Sigmaqsa = \Expect[ \qsaprobe_n\qsaprobe_n^\transpose]$.
\label{e:both_bf_bdds}
\end{subequations}

The error terms $\Deltaf$, $\Deltabarf$ are smooth functions of their arguments,  uniformly Lipschitz, and satisfying the following bounds:
 there exits $\bdd{t:Lip_qSGD} $ such that
  for all $\theta  \in \Re^d$:
\begin{equation}
\!\!\!\!\!
  \| \Deltabarf(\theta)\| \le  \max  \{  \| \Deltaf(\theta,\qsaprobe)\| : \qsaprobe\in\Re^d \}  \le \bdd{t:Lip_qSGD}  \min \{  \upepsilon_\theta  ,  \upepsilon_\theta^2 \}
\label{e:barbf_bdds}
\end{equation}\end{proposition}

\begin{proof}
The following identity is a step in the proof of the mean value theorem:
\[
\frac{1}{\upepsilon_\theta} \qsaprobe\Obj(\theta+\upepsilon_\theta\qsaprobe)
 = 
 \frac{1}{\upepsilon_\theta} \qsaprobe \Obj(\theta) 
 + 
 \qsaprobe \qsaprobe^\transpose \int^1_0  \nabla \Obj(\theta+\upepsilon_\theta\qsaprobe t)  \, dt
\]  
The integrand in the second term is Lipschitz in $\theta$,  since the composition of Lipschitz functions is Lipschitz.  The first term is Lipschitz since its gradient is bounded under the given assumptions.  This establishes \eqref{e:fLip}.
 
The representations \eqref{e:both_bf_bdds} are obtained from  the following identities, which are parts of the proof of the mean value theorem for $\Obj$ and $\nabla \Obj$, respectively: for $\theta \in \Re^d$,
\begin{align}
\Obj(\theta+\upepsilon_\theta\qsaprobe) 
&= 
\Obj(\theta) + \upepsilon_\theta \int^1_0 \qsaprobe^\transpose \nabla \Obj(\theta+\upepsilon_\theta\qsaprobe t)  \, dt
\label{e:first_mvt}
\\
\nabla \Obj(\theta+\upepsilon_\theta\qsaprobe t) 
&= 
\nabla \Obj(\theta) + t \upepsilon_\theta  \int^1_0 \nabla^2
 \Obj(\theta+\upepsilon_\theta \qsaprobe t r) \qsaprobe \, dr
\label{e:second_mvt}
\end{align}
 
Upon plugging \eqref{e:second_mvt} into \eqref{e:first_mvt}, adding and subtracting $\half \upepsilon_\theta^2 \qsaprobe^\transpose \nabla^2 \Obj(\theta) \qsaprobe$ to the right hand side, we obtain \eqref{e:f_Taylor} with 
\[
\Deltaf(\theta,\qsaprobe) 
= 
-\qsaprobe \upepsilon_\theta \iint    \qsaprobe^\transpose [\nabla^2 \Obj(\theta+\upepsilon_\theta\qsaprobe t r) - \nabla^2 \Obj(\theta)] \qsaprobe\,  t  \, dr \, dt
\]
This establishes  \eqref{e:f_Taylor}, and thence \eqref{e:meanflow_fTaylor} is obtained by taking expectations of both sides of \eqref{e:f_Taylor}. 
The bounds \eqref{e:barbf_bdds} follow from the representation for $\Deltaf$ and the Lipschitz assumption assumed for $\nabla^2 \Obj$.  
 \end{proof}

\subsection{Stability of the mean flow}    
Assumption (A4) imposes stability assumptions on the \textit{gradient flow}.   These carry over to the mean flow \eqref{e:meanflow} by application of  \Cref{t:Lip_qSGD}.    In particular, to establish limits such as in \eqref{e:BigSmallTargetBias}, theory from \cite{borchedevkonmey25} requires that the 
mean flow admit a linearization that is exponentially asymptotically stable. 
This along with exponential stability of the mean flow is established in the following.
\begin{proposition}
\label[proposition]{t:A*}
Under Assumptions (A1)--(A4) there is  $\upepsilon_\sbullet^a>0$ such that the following hold 
for any  $\upepsilon_\sbullet \in (0, \upepsilon_\sbullet^a]$.
\whamrm{(i)} The mean flow \eqref{e:meanflow} is exponentially asymptotically stable, with stationary point $\theta^*$.

\whamrm{(ii)}  
$A^*\eqdef \partial \barf\, (\theta^*)$ is a Hurwitz matrix.
 \end{proposition}

\begin{proof}
Part (i) is immediate from \Cref{t:Lip_qSGD} 
combined with \cite[Lemma 5.1]{kha02};
the representation \eqref{e:meanflow_fTaylor} 
combined with the bounds
\eqref{e:barbf_bdds}
implies that the mean flow vector field is a small Lipshcitz perturbation of the gradient flow for small $\upepsilon_\sbullet$.

Part (ii)  follows from \Cref{t:Lip_qSGD} which gives $A^* = -    \Sigmaqsa   \nabla^2   \Obj\, (\theta^\opt )  +O(\upepsilon_\sbullet)$.   
\end{proof}

\subsection{Geometric ergodicity} 
Another key assumption in \cite{borchedevkonmey25} is that $\bfPhi$ is a geometrically ergodic Markov chain on a general state space $\state$ with invariant measure $\uppi$, and satisfying  \textbf{(DV3)}:
\\
For  functions $V\colon\state\to\Re_+$,  $ W\colon\state\to [1, \infty)$, 
a small set $C$, $b>0$ and all $x\in\state$,
\begin{equation}
\begin{aligned}
\Expect\bigl[  \exp\bigr(  V(\Phi_{k+1})      \bigr)& \mid \Phi_k=x \bigr]  
\\
&
\le  \exp\bigr(  V(x)  - W(x) +  b \ind_C(x)  \bigl)
\end{aligned}
\label{e:DV3}
\end{equation}
The following conditions are also assumed:
\[\begin{aligned}
&S_W(r)  := \{ x :  W(x)\le r \} \quad \text{is either small or empty}
\\
 &\sup\{ V(x) :  x\in S_W(r) \}  <\infty \quad \text{for each $r\ge 1$ and} 
\\
&\lim_{r \to\infty} \sup_{x \in \state} \frac{L(x)}{\max\{r,W(x)  \}}  =0
% \Bigl(  \sup\Bigl\{  \frac{L(x) }{W(x)}  :   W(x)\ge n \Bigr \} \Bigr )  = 0 
\end{aligned}
\]
with $L$ the (state dependent) Lipchitz constant for $f$.  
 \Cref{t:Lip_qSGD} tells us that $L$ can be chosen independent of $x$.

The reader is referred to \cite{MT} for a proper definition of a small set.

In \Cref{t:DV3holds}, it is shown that any of the exploration choices imposed in (A3) satisfy the above conditions. 
We first establish a simpler property known as uniform ergodicity---we again refer to \cite{MT} or
\cite{borchedevkonmey25}
 for definitions.

\begin{lemma}
\label[proposition]{t:smallset}    Suppose that (A3ii) holds. Then,  $\bfPhi\eqdef \{ \Phi_n = (W_n,W_{n-1})^\transpose: n\geq 1  \}$ is a uniformly ergodic Markov chain on $\state=\Re^d\times\Re^d$.  Consequently, every measurable subset of $\state $ is small.
 \end{lemma}
 
\begin{proof}
For $n\geq 2$ and any measurable subset $A \subseteq X$, 
\[
\Prob(\Phi_n \in A \mid \Phi_0 = z )  = \uppi(A)  \, , \quad z \in \Re^d \times \Re^d
\]
which is far stronger than uniform ergodicity.
\end{proof}

The log moment generating function for $W_n$ is finite valued under (A3).    
Denote
 $\Lambda(\delta) \eqdef \log(\Expect[ \exp( \delta  \| W_n \| ) ]) $ for $\delta\in\Re$.
 We leave the proof of \Cref{t:DV3holds} to the reader: 
 for each part, it is straightforward to show that  \eqref{e:DV3} holds with equality.

 \begin{lemma}
\label[proposition]{t:DV3holds}    Suppose that  the assumptions of \Cref{t:Lip_qSGD} hold.
Then for any fixed $\delta_0>0$,   and with $b =\Lambda(\delta_0)+1$,   $C = \state$,

\whamrm{(i)} If (A3i) holds, the  Markov chain  $\{\Phi_n = W_n \}$ satisfies \eqref{e:DV3} with $V(x) =  \delta_0 \|  x \| $ and   $W(x) =  1 + V(x)$.

\whamrm{(ii)} If (A3ii) holds, the  Markov chain  $\{\Phi_n = (W_{n-1};W_n)^\transpose \}$ satisfies \eqref{e:DV3} with $V(w,w') = \delta_0[ \half \| w \| +  \| w' \| ] $ and 
\[
 W(w,w') =  1 + \half\delta_0[ \| w \| + \| w' \| ] 
\]
\qed
\end{lemma}

\begin{subequations}

\subsection{Variance analysis} 
 Consider any two functions   $g_1,g_2: \state \to \Re^d$  that may be real- or vector-valued. If each satisfies $\|g_i(x)\|^2 \le \exp(V(x))$ for each $x\in \state$ and $i = \{1,2\}$, their asymptotic covariance matrices $\SigmaCLT^{g_1},\SigmaCLT^{g_2}$ and cross covariance matrix $\SigmaCLT^{g_1,g_2}$ are given by 
\begin{align}
\!\!\!\!\! {}
&
\SigmaCLT^{g_i}  =   \sum_{k=-\infty}^{\infty} \Expect_\uppi[\tilg_i(\Phi_0) {\tilg_i(\Phi_k)}^\transpose] 
\\
&
\SigmaCLT^{g_1,g_2}   = \sum_{k=-\infty}^{\infty} \Expect_\uppi[\tilg_1(\Phi_0) {\tilg_2(\Phi_k)}^\transpose] 
\label{e:hollandCLT}
\end{align}
where $\tilg_i(x) = g_i(x) - \int g_i(x) \uppi(x)$ for $i = \{1,2\}$. 
We have for any pair of functions,
\begin{equation}
\SigmaCLT^{g_1+g_2} =  \SigmaCLT^{g_1} + \SigmaCLT^{g_1,g_2} +\SigmaCLT^{g_2,g_1}  +\SigmaCLT^{g_2}
\label{e:CLTsum}
\end{equation}
And,  under the assumptions to be imposed with have 
\[
\SigmaCLT^{g_i}  = \lim_{N \to \infty }    \frac{1}{N}    \Cov(S^{g_i}_N )
\]
where $S^{g_i}_N = \sum_{k=0}^N  g_i(\Phi_k) $ 
\end{subequations}

After explaining how the assumptions of the present paper satisfy the conditions imposed in \cite{borchedevkonmey25}, we turn to two results that will serve as foundation to the proof of \Cref{t:1SPSAcov}. The first conclusion in the following is \cite[Thm. 5]{borchedevkonmey25}, while the second is a step in its proof.

\begin{lemma}
\label[lemma]{t:Asympt_target}  
Suppose (A1)--(A4) hold. Then, the following limits hold:
\begin{romannum}
\item $\displaystyle \lim_{N \to \infty } N \Cov\big(\thetaPR_N \big) \eqdef  \SigmaPR = G\SigmaCLT^{\Delta^*}G^\transpose$

\item $\displaystyle \lim_{N \to \infty } N \Cov\big(\upbeta_N^{\barf}\big) = \SigmaCLT^{\Delta^*}$
\end{romannum}
where $G = -[\partial_\theta \barf(\theta^*)]^{-1}$ and $\SigmaCLT^{\Delta^*}$ is the asymptotic covariance of $\bfDelta^* \eqdef \{\Delta_n^* = f(\theta^*,\Phi_{n}):  n \geq 1\}$.
\qed
\end{lemma}

Proofs of the bounds in \eqref{e:BigSmallTargetBias} require the following decomposition of  $\bfDelta^* \eqdef \{\Delta_n^* = f(\theta^*,\Phi_{n}):  n \geq 1\}$  (recall \Cref{t:Asympt_target}).

\begin{lemma}
\label[lemma]{t:Delta_rep}
Under the assumptions of \Cref{t:Lip_qSGD}, the following representation holds: $\Delta^*_{n+1}  =  \upnu_{n+1} + \upomega_{n+1} + \uppsi_{n+1}$, in which
\[
\begin{aligned}
%\Delta^*_{n+1}  &=  \upnu_{n+1} + \upomega_{n+1} + \uppsi_{n+1}
%\\
%\text{where } 
\upnu_{n+1} & = -   \frac{1}{ \upepsilon_\sbullet}       \qsaprobe_{n+1} \Obj\, (\theta^* )
\\
\upomega_{n+1} &= [   \Sigmaqsa    -     \qsaprobe_{n+1} \qsaprobe_{n+1}^\transpose ]  \nabla   \Obj\, (\theta^* ) 
\\
\uppsi_{n+1} & = - \half \upepsilon_\sbullet    \qsaprobe_{n+1}^\transpose   \nabla^2   \Obj\, (\theta_n )   \qsaprobe_{n+1} \qsaprobe_{n+1}^\transpose  
+   \tilDeltaf(\theta^*,\qsaprobe_{n+1})
\end{aligned}
\label{e:Delta1SPSA}
\]
with $\tilDeltaf(\theta^*,\qsaprobe_{n+1}) \eqdef \Deltaf(\theta^*,\qsaprobe_{n+1}) -  \Deltabarf(\theta^*) $.
\qed
\end{lemma}

The next lemma implies that $\{\upomega_n\}$  has small impact on the asymptotic covariance of  $\{\Delta^*_n \}$.   

\begin{lemma}
\label[lemma]{t:cov_beta}  
Suppose (A1)--(A4) hold. 
Then, there is a constant $\bdd{t:cov_beta}$ such that 
\[
\trace(\SigmaCLT^{\upomega}) \leq \bdd{t:cov_beta}    \|   \nabla \Obj(\theta^*)  \|^2
\]
\end{lemma}
\begin{proof}
This is obvious in the case of (A3i)  (i.i.d.\ exploration).

For (A3ii) we must consider the cross-covariance.
The definition of $\{\upomega_n\}$ in \Cref{t:Delta_rep} implies the following identity,
\[\begin{aligned}
\Expect[\upomega_0 \upomega^\transpose_n ] &= \Expect[(\qsaprobe_0\qsaprobe^\transpose_0 - \Sigmaqsa) \nabla\Obj(\theta^*) \nabla\Obj^\transpose(\theta^*)(\qsaprobe_n\qsaprobe^\transpose_n - \Sigmaqsa)] 
\\
& = 0 \text{ for $|n|>1$}
\end{aligned}
\]

In view of the definition of the asymptotic covariance matrix in \eqref{e:hollandCLT}, the above equation yields 
$
\SigmaCLT^{\upomega} = \sum_{n=-1}^1 \Expect[\upomega_0 \upomega^\transpose_n ]
$ for which the following upper bound holds: $\trace(\SigmaCLT^{\upomega}) \leq 3\, \trace(\Expect[\upomega_0 \upomega^\transpose_0 ])$. 
The  definition of $\{\upomega_n\}$ in \Cref{t:Delta_rep} implies the upper bound
\[
\trace(\Expect[\upomega_0 \upomega^\transpose_0 ]) \leq \trace(\Expect[\qsaprobe_0\qsaprobe^\transpose_0 \|\qsaprobe_0 \|^2] - \Sigmaqsa) \| \nabla\Obj(\theta^*)\|^2
\]
which completes the proof. %

\end{proof}

\subsection{Proofs of the main results}
\label{s:Proofs!}

\begin{proof}[Proof of \Cref{t:1SPSAconv}]
Part (i) follows from \cite[Thm. 4]{borchedevkonmey25} with $\upepsilon_\sbullet^0 = \min\{ \upepsilon_\sbullet^a,\upepsilon_\sbullet^b \}$. Part (ii) follows (i) along with \Cref{t:Lip_qSGD} 
\end{proof}

\begin{proof}[Proof of  \Cref{t:1SPSAcov}]  
The proof of (i) begins with consideration of the representation in \Cref{t:Delta_rep}. Then, taking covariances of both sides yields
\[
\SigmaCLT^{\Delta^*} =\SigmaCLT^{\upnu} +  \SigmaCLT^{\upomega + \uppsi} +  \SigmaCLT^{\upnu,\upomega + \uppsi}  + \SigmaCLT^{\upomega + \uppsi,\upnu}
\]
Upon inspection of the definitions of $\{\upnu_n, \upomega_n,\uppsi_n\}$ in \Cref{t:Delta_rep}, it follows that the  term  $\SigmaCLT^{\upnu}$ dominates the asymptotic variance. This, along with  \Cref{t:Asympt_target}~(ii), completes the proof of (i).

 Under zig-zag exploration, the sequence $\bfqsaprobe$ is telescoping. The fact that telescoping sequences have zero asymptotic covariance, along with the representation in \Cref{t:Delta_rep},  justifies the identity
\[
\begin{aligned}
\SigmaCLT^{\Delta^*} = \SigmaCLT^{\Delta^* - \upnu} &= \SigmaCLT^{\upomega + \uppsi}
\\
&=\SigmaCLT^{\upomega} +  \SigmaCLT^{\uppsi} +  \SigmaCLT^{\upomega,\uppsi}  + \SigmaCLT^{\uppsi,\upomega}
\end{aligned}
\]
in which the last equality follows from \eqref{e:CLTsum}. 

The right hand side of the above representation is bounded in $\upepsilon_\sbullet$.
Moreover, the Cauchy-Schwarz inequality implies the bound
\[
\trace(\SigmaCLT^{\upomega,\uppsi}  + \SigmaCLT^{\uppsi,\upomega}) \leq 2 \sqrt{\trace(\SigmaCLT^{\upomega})} \sqrt{\trace(\SigmaCLT^{\uppsi})}
\]

In view of the definition of $\{ \psi_n \}$ in \Cref{t:Delta_rep}, it follows that $\trace(\SigmaCLT^{\uppsi} ) = O(\upepsilon_\sbullet^2)$. 

It remains to bound $\SigmaCLT^{\upomega}$. The approximation in \eqref{e:meanflow_fTaylor} and the upper bound in \eqref{e:barbf_bdds} imply that $\| \nabla \Obj(\theta^*) \|^2 =O(\upepsilon_\sbullet^4) $, which in combination with \Cref{t:cov_beta}, yields $\trace(\SigmaCLT^{\upomega} ) = O(\upepsilon_\sbullet^4)$. This combined with \Cref{t:Asympt_target}~(ii) completes the proof of (ii).   
\end{proof}

 \def\urls#1{{\scriptsize\url{#1}}}

	\clearpage

 \bibliographystyle{abbrv}
\bibliography{strings,markov,q,QSA,bandits}

  \end{document}